\documentclass{article}
\usepackage{graphicx}
\usepackage{amsthm, amsmath, amssymb, tikz, bm}
\usepackage{mathtools}
\usepackage{xifthen}
\usepackage[normalem]{ulem}

\usepackage[margin=1.2in]{geometry} 

\usepackage[shortlabels]{enumitem}
\usepackage{todonotes}
\usepackage[colorlinks=true,
linkcolor=blue,citecolor=blue,
urlcolor=blue]{hyperref}

\allowdisplaybreaks

\usetikzlibrary{calc,shapes, backgrounds}

\newtheorem{theorem}{Theorem}
\newtheorem{lemma}[theorem]{Lemma}

\newtheorem{definition}[theorem]{Definition}
\newtheorem{claim}[theorem]{Claim}
\newtheorem{constrain}[theorem]{Constraint}

\newtheorem{conjecture}[theorem]{Conjecture}

\newtheorem{question}[theorem]{Question}
\newtheorem{observation}[theorem]{Observation}
\newtheorem{constr}[theorem]{Construction}

\title{An Ore-type theorem for $[3]$-graphs}
\author{Yupei Li
\thanks{University of South Carolina, Columbia, SC 29208, ({\tt yupei@email.sc.edu}).}
 \and Linyuan Lu \thanks{University of South Carolina, Columbia, SC 29208,
({\tt lu@math.sc.edu}).
}
\and Ruth Luo \thanks{University of South Carolina, Columbia, SC 29208, ({\tt ruthluo@sc.edu}).}
}

\newcommand{\LL}[1]{{\textcolor{blue}{#1}}}

\newcommand{\cH}{{\cal H}}
\newcommand{\cP}{{\cal P}}
\newcommand{\cC}{{\cal C}}
\newcommand{\cG}{{\cal G}}
\begin{document}
\maketitle

\begin{abstract}
Ore's Theorem states that if $G$ is an $n$-vertex graph and every pair of non-adjacent vertices has degree sum at least $n$, then $G$ is Hamiltonian. 

A $[3]$-graph is a hypergraph in which every edge contains at most $3$ vertices.
In this paper, we prove an Ore-type result on the existence of Hamiltonian Berge cycles in $[3]$-graph $\cH$, based on the degree sum of every pair of non-adjacent vertices in the $2$-shadow graph $\partial \cH$ of $\cH$.
Namely, we prove that there exists a constant $d_0$ such that for all $n \geq 6$, if a $[3]$-graph $\cH$ on $n$ vertices satisfies that every pair $u,v \in V(\cH)$ of non-adjacent vertices has degree sum $d_{\partial \cH}(u) + d_{\partial \cH}(v) \geq n+d_0$, then $\cH$ contains a Hamiltonian Berge cycle.  Moreover, we conjecture that $d_0=1$ suffices.
\end{abstract}

\section{Introduction}


 A {\em Hamiltonian cycle (path)} in a graph $G$ is a cycle (path) that covers all of the vertices in $G$. We say $G$ is {\em Hamiltonian} if it contains a Hamiltonian cycle. It is well known that determining if a graph is Hamiltonian is an NP-complete problem. Finding sufficient conditions for Hamiltonicity in graphs is one of the most well-studied fields in combinatorics. 
The first extremal result for Hamiltonicity was due to Dirac.

\begin{theorem}[Dirac's Theorem~\cite{D}]
Let $G$ be a graph on $n\geq 3$ vertices. If $\delta(G) \geq n/2$, then $G$ contains a Hamiltonian cycle.
\end{theorem}
This was later strenghthened by Ore. 
\begin{theorem}[Ore's Theorem~\cite{Ore}]
Let $G$ be a graph on $n\geq 3$ vertices. If $G$ has the property that 
for all pairs of non-adjacent vertices $(u,v)$ $d(u)+d(v)\geq n$, then
$G$ contains a Hamiltonian cycle.
\end{theorem}
Note that Ore's Theorem directly implies Dirac's Theorem since a graph with minimum degree $n/2$ trivially has pairwise degree sums of at least $n$. 

We will consider analogs of Hamiltonian cycles and paths in hypergraphs. We first present some definitions related to hypergraphs.
A hypergraph $\cH$ is a family of subsets of a ground set. We refer to these subsets as the edges of $\cH$ and the elements of the ground set as the vertices of $\cH$, and use $E(\cH)$ and $V (\cH)$ to denote these sets, respectively. We say $\cH$ is $r$-uniform (an $r$-graph, for short) if every edge of $\cH$ contains exactly $r$ vertices. 
Moreover, if $R$ is a set of integers, an $R$-graph is a hypergraph in which every edge has size in $R$.

Given a hypergraph $\cH$ and vertices $u,v \in V(\cH)$,
the {\em degree} $d_{\cH}(u)$ of $u$ is the number of edges containing $u$, and the {\em co-degree} $d_{\cH}(u,v)$ of $u$ and $v$ is the number of edges containing both of $u$ and $v$. 
We say $u$ and $v$ are {\em adjacent} if their co-degree is at least $1$, otherwise they are non-adjacent.
The {\em neighborhood} of a vertex $u$ is the set $N_{\cH}(u) = \{v: d_{\cH}(u,v) \geq 1\}$. In the above notations when $\cH$ is clear under context, we will omit the subscripts. The \textit{$2$-shadow}(or \textit{shadow}) of $\cH$, denoted by $\partial \cH$, is the simple $2$-uniform graph with $V(\partial \cH)=V(\cH)$ and $uv\in E(\partial \cH)$ if and only if $\{u,v\}\subseteq h$ for some $h\in E(\cH)$ (i.e., $d_{\cH}(u,v)\geq 1$). 
 
A \textit{Berge path} of length $t$ is a collection of $t$ distinct edges $e_1, e_2, \ldots, e_{t} \in E(\cH)$ and $t+1$ distinct vertices $v_1, v_2, \ldots, v_{t+1}\in V(\cH)$ such that $\{v_i, v_{i+1}\} \subseteq e_i$ for each $i\in [t]$. We can view a Berge path as an alternating sequence of vertices and edges $v_1,e_1, v_2, \ldots, e_t, v_{t+1}$ such that each edge contains the incident vertices in the sequence.

Although the edges $e_i$ in $\cP$ may contain vertices outside of $\{v_1, \ldots, v_{t+1}\}$, when we refer to the vertex set of $\cP$ we ignore such vertices and simply define $V(\cP) =\{v_1, \ldots, v_{t+1}\}$ and $E(\cP) = \{e_1, \ldots, e_{t}\}$.

Similarly, a \textit{Berge} cycle of length $t$ is a collection of $t$ distinct edges $e_1, e_2, \ldots, e_t \in E(\cH)$ and $t$ distinct vertices $v_1, v_2, \ldots, v_t \in V(\cH)$ such that $\{v_i, v_{i+1}\} \subseteq e_i$ for every $i\in [t]$ where $v_{t+1} \equiv v_1$. We define $V(\cC) = \{v_1, \ldots, v_t\}$ and $E(\cC) = \{e_1, \ldots, e_t\}$.

A \textit{Hamiltonian Berge cycle (path)} of a hypergrah $\cH$ is a Berge cycle (path) covering all vertices in $\cH$ (i.e., of length $n$ or $n-1$ respectively). We say $\cH$ is {\em Berge-Hamiltonian} if it contains a Hamiltonian Berge cycle.


Coulson and Perarnau proved the following result which implies a version of Dirac's Theorem for sufficiently large hypergraphs.

\begin{theorem}[Coulson-Perarnau~\cite{CP}]
There exists $\mu > 0$ and $n_0 \in \mathbb N$ such that if $n \geq n_0$ and $G$ is an $n$-vertex graph with $\delta(G) \geq n/2$, then any coloring of $E(G)$ in which each color is used at most $\mu n$ times contains a rainbow Hamiltonian cycle.
\end{theorem}

To see the application to hypergraphs, for $r = o(\sqrt{n})$, suppose $\cH$ is a $r$-graph whose shadow satisfies $\delta(\partial \cH) \geq n/2$. Given a hyperedge $e \in E(\cH)$, color all edges 
 in $\partial \cH$ between pairs in $e$ with a common color unique to $e$ (if there exists pairs in multiple edges break ties arbitrarily). Then a rainbow Hamiltonian cycle in $\partial \cH$ corresponds to a Hamiltonian Berge cycle as we can map each colored edge to its corresponding hyperedge. In particular, this implies the following Dirac-type Theorem for hypergraphs.

 \begin{theorem}[Coulson-Perarnau~\cite{CP}]\label{dirac2}Let $r = o(\sqrt{n})$, and suppose $\cH$ is an $n$-vertex, $r$-uniform hypergraph. If $\delta(\cH) \geq {\lfloor (n-1)/2 \rfloor \choose r-1} + 1$, then $\cH$ contains a Hamiltonian Berge cycle.
 \end{theorem}

Later, Kostochka, Luo, and McCourt extended this result to all combinations of $r$ and $n$.

 \begin{theorem}[Kostochka, Luo, McCourt~\cite{KLM}]\label{dirac3}Let $3\leq r \leq n$, and suppose $\cH$ is an $n$-vertex, $r$- uniform hypergraph. If
 (1) $r \leq (n-1)/2$ and $\delta(\cH) \geq {\lfloor (n-1)/2 \rfloor \choose r-1} + 1$, or if (2) $r \geq n/2$ and $\delta(\cH) \geq r$, then $\cH$ contains a Hamiltonian Berge cycle.
 \end{theorem}

For more related results on Berge paths and cycles in hypergraphs, see~\cite{BGHS,CEP,CP,EGMSTZ,FKLdirac,KL1,KLM,MHG} for instance.

Our goal in this paper is to prove a version of Ore's Theorem for hypergraphs. Recall that Ore's Theorem states that if every pair of nonadjacent vertices in an $n$-vertex graph have degree sum at least $n$, then there exists a Hamiltonian cycle.
In generalizing to hypergraphs, one must choose an appropriate analog of degree sum. This turns out to be nontrivial depending on the uniformity or size of the edges in the hypergraph. Consider the following example.

\begin{constr}\label{const1} For an integer $r \geq 3$ let $\cH_r'$ denote the $r$-graph on $n> r$ vertices consisting of an $n-1$ clique $K_{n-1}^{(r)}$ and another vertex $v$ contained in exactly $1$ edge.
\end{constr}

Observe that $d_{\cH_r'}(x)+ d_{\cH_r'}(y) \geq {n-2 \choose r-1} +1$ for all nonadjacent $x$ and $y$ in $\cH_r'$, and $\cH_r'$ is not Berge-Hamiltonian because it contains a vertex with degree $1$.
Using this notion of degree, the degree sum needed for an Ore-type theorem is at least ${n-2 \choose r-1}+2$. Unlike the graph case however, this Ore result would not imply the analogous Dirac-type theorem for hypergraphs, Theorems~\ref{dirac2} and~\ref{dirac3}, when $n$ is large and $r$ is fixed as ${n-2 \choose r-1} + 2$ is much bigger than $2 ({\lfloor (n-1)/2 \rfloor \choose r-1} + 1)$.

We consider another natural notion of degree---the number of neighbors of a vertex (equivalently, the degree in $\partial \cH$).
We propose the following conjecture regarding vertex degrees in the shadow of a $[3]$-uniform hypergraph. Here, $[3]$ refers to the set $\{1, 2, 3\}$. Since edges of size $1$ do not contribute to Berge cycles, our focus is on $\{2,3\}$-uniform hypergraphs. However, for simplicity, we continue to refer to them as $[3]$-uniform hypergraphs, or simply $[3]$-graphs.

\begin{conjecture}\label{conj:ore}  Let $n\geq 6$, and let $\cH$ be a $[3]$-graph on $n$ vertices such that every pair $u,v \in V(\cH)$ of non-adjacent vertices satisfies $d_{\partial \cH}(u) + d_{\partial \cH}(v) \geq n+1$. Then $\cH$ contains a Hamiltonian Berge cycle.
\end{conjecture}

{\bf Remark 1:}  $n+1$ is the best possible: The graph $\cH_3'$ from Construction~\ref{const1} is not Berge Hamiltonian and has $d_{\partial \cH_3'}(u) + d_{\partial \cH_3'}(v) =n$ for any vertex $u$ not adjacent to $v$. 

In this paper, we prove a weaker version of the conjecture. 

\begin{theorem}\label{mainthm}
There exists a constant $d_0 \geq 1$ such that the following statement holds. Let $n\geq 6$, and let $\cH$ be a $[3]$-graph on $n$ vertices such that every pair $u,v \in V(\cH)$ of non-adjacent vertices satisfies $d_{\partial \cH}(u) + d_{\partial \cH}(v) \geq n+d_0$. Then $\cH$ contains a Hamiltonian Berge cycle.
\end{theorem}

We will prove Theorem \ref{mainthm} with $d_0=65$. The constant $n=6$ is best possible: consider
a 3-uniform hypergraph $(V,E)$ (see \cite{LW}) on $5$-vertex and $4$-edges: $V=[5]$, $E=\{\{1,2,3\}, \{1,2,4\}, \{1,2,5\}, \{3,4,5\}\} .$ 
All pairs of vertices are adjacent. Hence the weak Ore Condition with any $d_0$ trivially holds. However, this hypergraph only has $4 < n$ edges and therefore cannot contain a Hamiltonian Berge cycle.

A hypergraph $\cH$ is called {\em covering} if every pair of vertices has co-degree at least 1 (i.e., $\partial \cH$ is a complete graph).   Lu and Wang \cite{LW} proved the following theorem.
\begin{theorem}[Lu-Wang\cite{LW}] \label{thm: Lu-Wang}
Every covering $[3]$-graph $\cH$ on $n \geq 6$ vertices  contains a Hamiltonian Berge cycle.
\end{theorem}

We only need to prove Theorem \ref{mainthm} for $n\geq d_0+4$:  if $6\leq n\leq d_0+3$ and some pair of vertices $u, v$ are non-adjacent, then $d_{\partial \cH}(u) + d_{\partial\cH}(v) \leq n-2 + n-2 \leq n + d_0 -1$. Thus $\cH$ must be covering, and by Theorem \ref{thm: Lu-Wang},
 $\cH$ contains a Hamiltonian Berge cycle.
Thus for the remainder of the paper, we will assume $n\geq d_0+4$.

\subsection{Proof idea and outline}
Our proof is inspired by classical proof arguments for Hamiltonian cycles in graphs. We give the following motivating example. Suppose $C = v_1, \ldots, v_k, v_1$ is a longest cycle in a graph $G$. If $G$ is not Hamiltonian, then there exists a vertex $u \notin V(C)$. The vertex $u$ cannot have neighbors that appear consecutively in $C$ otherwise we could insert $u$ between them to obtain a longer cycle. Similarly, if $v_i, v_j \in N(u)$, then $v_{i+1}$ and $v_{j+1}$ cannot be adjacent otherwise we can find a longer cycle.

We would like to transfer these ideas to hypergraphs, but we run into various problems due to the larger edges. For instance, if $\cC = v_1, e_1, \ldots, v_k, e_k, v_1$ is a longest Berge cycle and $u \notin V(\cC)$ is adjacent to $v_{i}$ and $v_{i+1}$, then we would {\em like} to add $u$ to the Berge cycle between these vertices but this is not always possible. For example if $u$ is only adajacent to $v_{i}$ and to $v_{i+1}$ via the edge $e_i$, or if $e_{i-1}$ is the only edge connecting $u$ to $v_i$, then we may not be obtain a longer Berge cycle. 

In Section 2, we introduce some definitions and prove that $\cH$ has a Berge cycle $\cC$ of length at least $n-2$. For each vertex $u \notin V(\cC)$, we construct a subset $U_u \subseteq V(\cC)$ called a {\em usable set}. This set has the property that $\{v_{i-1}: v_i \in U_u\}$ is a large subset of the neighorhood of $u$ and for any pair $v_i, v_j$ of such neighbors, there exists distinct edges connecting $u$ to each of these vertices.

We conclude Section 2 with our main lemma, Lemma~\ref{lem:main lem}, which roughly states that all vertices outside of longest Berge cycles have small degree, all but at most $6$ vertices in the usable set $U_u$ have many neighbors {\em but not too many}, and every pair of such vertices creates several {\em bridges} (to be defined in the next section) corresponding to edges in the Berge cycle.

In Section 3, we complete our proof of Theorem~\ref{mainthm} by applying our main lemma. We begin with a longest Berge cycle $\cC$ and show that $\cC$ has several vertices with large degree that produce many bridges. On the other hand, the total number of bridges is at most the number of edges in $\cC$ which yields a contradiction. 

Throughout the proof, we will assume that every pair $u,v \in V(\cH)$ of non-adjacent vertices has degree sum $d_{\partial \cH}(u) + d_{\partial\cH}(v) \geq n + d_0$ for some absolute constant $d_0$. We collect a series of constraints we need $d_0$ to satisfy simultaneously. At the end of our proof we conclude that our choice $d_0 = 65$ suffices.

\section{Notations and lemmas}

For $u,v \in V(\cH)$, we denote by $E(u,v) = \{e \in E(\cH): \{u,v\} \subseteq e\}$ the set of edges in $\cH$ containing both $u$ and $v$. Recall that $\partial \cH$ is the $2$-shadow of $\cH$.
By definition, $N_{\partial \cH}(v) =N_{\cH}(v)$, so we simply refer to this set by $N(v)$. We also have $d_{\partial \cH} (u) = |N(v)|$. Given a Berge cycle $\cC$, denote $N_{\cC}(u) = N(u) \cap V(\cC)$ and $d_{\cC}(u) = |N_{\cC}(u)|$. 


Suppose $\cH$ is an edge-maximal counterexample to Theorem~\ref{mainthm}. That is, $\cH$ is a $[3]$-graph which satisfies $d_{\partial \cH}(u)+d_{\partial \cH}(v) \geq n + d_0$ for all nonadjacent $u,v \in V(\cH)$, $\cH$ contains no Hamiltonian Berge cycle, but adding any edge to $\cH$ creates a Hamiltonian Berge cycle.  Thus $\cH$ contains a Hamiltonian Berge path, 
\[\cP = v_1,e_1, v_2, e_2, \ldots, e_{n-1}, v_n.\]

\begin{claim}\label{n-1n-2} $\cH$ contains a Berge cycle of length $n-1$ or $n-2$.
\end{claim}
\begin{proof}
First consider the case that $v_1$ and $v_n$ are adjacent. Let $e \in E(v_1,v_n)$. If $e \notin E(\cP)$, then \[v_1,e_1, v_2, e_2, \ldots, v_n, e, v_1\] is a Hamiltonian Berge cycle, a contradiction to $\cH$ being an edge-maximal counterexample. Thus, $e \in E(\cP)$. Since $|e|\leq 3$, we must have either $e = e_{1} = \{v_1, v_2, v_n\}$ or $e = e_{n-1} = \{v_1, v_{n-1}, v_n\}$.  Without loss of generality, assume $e=e_1=\{v_1,v_2,v_n\}$. Then we observe that $v_2,e_2,\ldots,v_n,e,v_2$ is a Berge cycle of length $n-1$, as desired.

Now suppose that $v_1$ and $v_n$ are non-adjacent. Then $d_{\partial \cH}(v_1)+d_{\partial \cH}(v_n)\geq n+d_0$. Define $N^-(v_1)=\{v_{i-1}: v_i \in N(v_1)\}$, and observe that $N^-(v_1) \subseteq \{v_1, \dots, v_{n-2}\}$ and $N(v_n) \subseteq \{v_2, \ldots, v_{n-1}\}$, so $|N^-(v_1) \cup N(v_n)| \leq n-1$. 

Since $|N^-(v_1)|+|N(v_n)|=d_{\partial \cH}(v_1)+d_{\partial \cH}(v_n)\geq n+d_0$, there exists  at least $d_0+1$ indices $i \in [n]$ such that $v_{i-1} \in N^-(v_1) \cap N(v_n)$. That is $v_i \in N(v_1)$ and $v_{i-1} \in N(v_n)$. For such an index $i$, we say $v_1$ and $v_n$ have a {\em crossing} at $i$. 

For such $i$, take $e \in E(v_1,v_i)$ and $ e' \in E(v_n,v_{i-1})$. Note that since $v_1$ and $v_n$ are non-adjacent, $e \neq e'$.  If $e,e' \notin E(\cP)-\{e_{i-1}\}$, then \[v_1, e, v_i, e_i, \ldots ,v_n, e', v_{i-1}, e_{i-2}, \ldots,v_1\] is a Hamiltonian Berge cycle, a contradiction.

So we may assume that for every $i$ such that $v_1$ and $v_n$ have a crossing at $i$, at least one of $e$ and $e'$ belong to $E(\cP)-\{e_{i-1}\}$. Note that since $|e|, |e'| \leq 3$, if $e \in E(\cP) - \{e_{i-1}\}$ then $e \in \{e_1, e_{i}\}$, and if $e' \in E(\cP) - \{e_{i-1}\}$, then $e' \in \{e_{n-1}, e_{i-2}\}$. Since there are at least $d_0+1 \geq 3$ indices $i$ to choose from, we may choose $i$ and edges $e, e'$ such that $e \neq e_1$ and $e' \neq e_{n-1}$.
Here we assume 
\begin{constrain}\label{constrain0}
    $d_0 \geq 2$.
\end{constrain}
Without loss of generality, we may assume $e = e_i$. If $e' \neq e_{i-2}$, we obtain the Berge cycle \[v_1, e_i, v_{i+1}, e_{i+1}, \ldots, v_n, e', v_{i-1}, e_{i-2}, \ldots, v_1\] which has length $n-1$. If $e' = e_{i-2}$, we instead obtain the Berge cycle of length $n-2$ \[v_1, e_{i}, v_{i+1}, e_{i+1}, \ldots, v_n, e_{i-2}, v_{i-2}, e_{i-3}, \ldots, v_1.\]

\end{proof}

\begin{definition}
We say a Berge cycle $\cC$ is a {\bf maximal Berge cycle} if there does not exist a Berge cycle $\cC'$ with $V(\cC) \subsetneq V(\cC').$    
\end{definition}

Let $\cC=v_1,e_1, \ldots, v_\ell, e_\ell, v_1$ be a Berge cycle. When referring to vertices $v_i$ in $\cC$, all indices will be written modulo $|V(\cC)|$. We view $\cC$ as oriented {\em clockwise}. For a set $A \subseteq V(\cC)$ and an integer $k \in \mathbb Z$, we denote the set $A^{+k} = \{v_{i+k}: v_i \in A\}$. For simplicity, we will use $A^+$ and $A^-$ to denote the special cases $k=1$ and $k=-1$, respectively.

For three distinct vertices $v_i, v_j, v_k \in V(\cC)$, we write
\[v_i \to v_j \to v_k\]
to mean that $v_j$ appears between $v_i$ and $v_k$ in the clockwise segment of $\cC$ from $v_i$ to $v_k$. 
For example, $v_1 \to v_5 \to v_6$ and $v_{6} \to v_8 \to v_1$. We similarly define the relation $v_{i_1} \to v_{i_2} \to \ldots \to v_{i_t}$ for any set of $t \geq 3$ vertices in $V(\cC)$.


\begin{definition}Let $\cC=v_1, e_1, \ldots, v_\ell, e_{\ell}, v_1$ be a Berge cycle, and let $u \in V(\cH) - V(\cC)$. 
A set $U_u \subseteq V(\cC)$ is {\bf usable} for $u$ if 
\begin{enumerate}
    \item for any $v_i\in U_u$, 
    there exists an edge $e(u,v_{i-1}) \in E(u, v_{i-1}) - \{e_{i-2}\}$. In particular, this implies $v_{i-1} \in N(u)$.
    \item for any pair of vertices $v_i, v_j\in U_u$, we have $|i-j| \geq 2$. 
   \item for any pair of vertices $v_i, v_j\in U_u$, there exist edges $e(u,v_{i-1}) \in E(u, v_{i-1}) - \{e_{i-2}\}$ and $e(u, v_{j-1})\in E(u, v_{j-1}) - \{e_{j-2}\}$ such that $e(u,v_{i-1}) \ne e(u, v_{j-1})$.
\end{enumerate}
\end{definition}

\begin{lemma}\label{lem: usableset}For any $u \in V(\cH) - V(\cC)$, there exists a set $U_u$ that is usable for $u$ with $|U_u| \geq d_\cC(u)/6$.
\end{lemma}
\begin{proof}
 First we prove the following. 
\begin{equation}\label{U0}\mbox{There exists a set $U_0$ of size at least $d_\cC(u)/3$ satisfying parts 1 and 2 in the definition of usable.}
\end{equation}

Call a vertex $v_r \in N_\cC(u)$ {\em bad} if $E(u, v_{r})=\{e_{r-1}\}$, otherwise call it {\em good}. If there are no bad vertices, then choose a largest subset $N \subseteq N_{\cC}(u)$ that does not contain a pair of consecutive vertices in $\cC$. If $N_{\cC}(u) \neq V(\cC)$, then $|N| \geq d_{\cC}(u)/2$, otherwise $|N| \geq \lfloor |\cC|/2 \rfloor$. The set $U_0 = \{v_r: v_{r-1} \in N\}$ satisfies~\eqref{U0}.

Now suppose there are $b \geq 1$ bad vertices and $g$ good vertices in $N_\cC(u)$.
For a bad vertex $v_r$, we have $e_{r-1}=\{v_{r-1},v_r,u\}$. Therefore $v_{r-1}$ is good because $e_{r-1} \in E(u, v_{r-1})-\{e_{r-2}\}$. Letting $U' = \{v_{r-1}: v_r$ is bad$\}$, we obtain that $|U'| = b$, and $U'$ contains no pair of consecutive vertices in $\cC$, otherwise some vertex is both bad and good.

Let $N' = N_\cC(u) - \{\{v_r, v_{r-1}, v_{r-2}\}: v_r \text{ is bad}\}$. 
Then we can find a subset $N''\subseteq N'$ of size at least $|N'|/2$ that contains no consecutive vertices in $\cC$.  Observe $U_0 = \{v_r: v_{r-1} \in U' \cup N''\}$ satisfies parts 1 and 2 in the definition of usable. We will show that $|U_0| = |U'| + |N''| \geq d_{\cC}(u)/3$. 

If $|U'| \geq d_{\cC}(u)/3$ then we are done, so suppose $|U'| = b < d_{\cC}(u)/3$. 
Then
\[|U'| + |N''| \geq b + \frac{d_{\cC}(u) - 3b}{2} > \frac{d_{\cC}(u)}{3}.\]This proves~\eqref{U0}.
Next suppose some pair $v_{i}, v_{j} \in U_0$ does not satisfy part $3$ in the definition of usable. The sets $E(u, v_{i-1}) - \{e_{i-2}\}$ and $E(u, v_{j-1}) - \{e_{j-2}\}$ are nonempty. We say $\{i,j\}$ is a bad pair if each of these sets contain a single edge, and the edges are the same (i.e., $\{u, v_{i-1}, v_{j-1}\}$). The bad pairs must then be disjoint, and we arbitrarily delete one vertex from each bad pair from $U_0$. The resulting set $U_u$ is usable for $u$ with size $|U_u| \geq |U_0|/2 \geq d_{\cC}(u)/6$.
\end{proof}


\begin{claim}\label{uadjacent} Suppose $\cC$ is a maximal Berge cycle and $u \in V(\cH \setminus \cC)$. For any $v_i, v_j \in U_u$, if $v_i$ and $v_j$ are adjacent, then then $E(v_i, v_j) \subseteq \{e_i, e_j\}$.
\end{claim}

\begin{proof}Suppose $f \in E(v_i,v_j)$ and $f \notin \{e_i, e_j\}$. Let $e(v_{i-1}, u) \in E(v_{i-1}, u)-\{e_{i-2}\}$ and $e(v_{j-1}, u) \in E(v_{j-1}, u)-\{e_{j-2}\}$ be distinct edges. Then $f \neq e(v_{i-1}, u)$ otherwise $\{v_{i-1}, v_i, v_j, u\} \subseteq f$, contradicting that $|f| = 3$. Here we use that $|j - i| \geq 2$ by the definition of usable. Similarly, $f \neq e(v_{j-1}, u)$. 

The Berge cycle
\[v_j, e_j, \ldots, v_{i-1}, e(v_{i-1}, u), u, e(v_{j-1}, u), v_{j-1}, e_{j-2}, \ldots, v_i, f, v_j\] has vertex set $V(\cC)\cup \{u\}$, contradicting the maximality of $\cC$. (See Figure \ref{fig:lem1path}.) 
\end{proof}


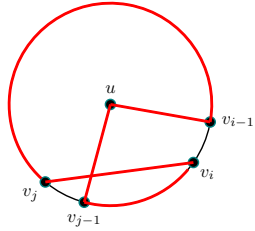
\begin{figure}[htb]
\centering
\resizebox{3.5cm}{!}{
\begin{tikzpicture}[thick]
\def\r{2} 
\draw (0,0) circle(\r)
(0,0)--(-10:\r)
(0,0)--(-105:\r)
(-130:\r)--(-35:\r);

\path
(0,0) node [above=1mm] (v1) {$u$}
(-35:\r) node[below right]{$v_i$}
(-130:\r) node[below left]{$v_j$}
(-10:\r) node[right=1mm]{$v_{i-1}$}
(-105:\r)  node[below=1mm]{$v_{j-1}$};

\foreach \point in {(0,0),(-35:\r),(-130:\r),(-10:\r),(-105:\r)}
\draw[teal,fill=black] \point circle(.1);
\draw[red,ultra thick] (-105:\r) arc (-105:-40:2);
\draw[red,ultra thick] (-10:\r) arc (-10:227:2);

\draw[red,ultra thick] (0,0) -- (-10:\r);

\draw[red,ultra thick] (0,0) -- (-105:\r);
\draw[red,ultra thick] (-130:\r)--(-35:\r) ;
\end{tikzpicture}}
\caption{$v_i, v_j \in U_u$ and $f \notin \{e_i, e_j\}$.}\label{fig:lem1path}
\end{figure}

We divide the vertices of $\cH$ into the following sets based on their degrees. We refer to these sets as the super small, small, big, and super big vertices, respectively. 
The constants $\gamma_1$ and $\gamma_2$ in the definitions will be chosen later.

\[S'(\cC)=\{v: d_\cC(v)< \frac{n}{4} +\gamma_1\}, \qquad 
S(\cC)=\{v: d_\cC(v)< \frac{n+d_0}{2}-3\}\]
\[B(\cC)=\{v: d_\cC(v) \geq \frac{n+d_0}{2}-3\}, \qquad B'(\cC)=\{v: d_\cC(v)\geq \frac{3n}{4}+\gamma_2\}.\]

When $\cC$ is clear from the context, we let $S'(\cC)=S', S(\cC)=S, B(\cC)=B, B'(\cC)=B'$. 

We require $S'\subseteq S$ and $B' \subseteq B$ for $n \geq d_0+4$, that is
\[\frac{n}{4} +\gamma_1 \leq \frac{n+d_0}{2}-3 \qquad \mbox{and} \qquad \frac{n+d_0}{2}-3 \leq \frac{3n}{4}+\gamma_2.\]
It is sufficient to assume
\begin{constrain}\label{constrain00}
$4 \gamma_1+8 \leq 3 d_0$ and $d_0 \leq 4 \gamma_2+16$.
\end{constrain}

Observe that for any pair of small vertices, $u, v \in S (\supseteq S')$, if $|V(\cC)| \geq n-3$, 

\[d_{\partial\cH}(u) + d_{\partial\cH}(v) \leq d_{\cC}(u) + 3 + d_{\cC}(v) + 3 < \frac{n+d_0}{2} + \frac{n+d_0}{2}.\]
It follows that
\begin{equation}\label{small}
\mbox{if $|V(\cC)| \geq n-3$, then all vertices in $S$ are pairwise adjacent.}
\end{equation}

\begin{definition}For indices $i,j,k$, we say an edge $e$ is an {\bf $(i,j;k)$-bridge} if
\begin{itemize}
    \item $v_i \to v_k \to v_j$ and $e = e_{k+1} = \{v_i, v_{k+1}, v_{k+2}\}$ or $e =e_{k-1}= \{v_j, v_{k}, v_{k-1}\}$, or
    \item  $v_j \to v_k \to v_i$ and $e =e_{k-1}= \{v_i, v_k, v_{k-1}\}$ or $e =e_{k+1}= \{v_j, v_{k+1}, v_{k+2}\}$.
\end{itemize}



The {\bf bridges produced by the pair $\{v_i, v_j\}$} is the set $R(v_i,v_j)$ of all the $(i,j;k)$-bridges for any $k \in [n] - \{i,j\}$.
\end{definition}

Since each bridge corresponds to an edge in $\cC$, if $\cC$ is not Hamiltonian, 
\begin{equation}\label{numbridges}
    \bigg|\bigcup_{\{v_i,v_j\}\subset V(\cC)} R(v_{i},v_{j})\bigg| \leq |E(\cC)| \leq n-1.
\end{equation}



The same bridge $e$ with $v_i \in e$ may be produced by multiple pairs $\{v_i,v_j\}$, $j\in [n]- \{i\}$. Hence we have the following observation.
\begin{observation}\label{bridgeproduced}
For any set $X \subseteq V(C)$, a bridge $e$ can be produced by at most $|X|-1$ pairs of vertices in $X$.
\end{observation}


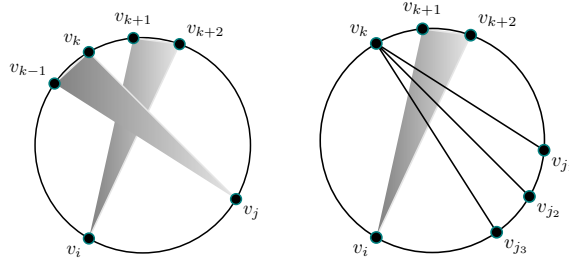
\begin{figure}[htb]
\centering
\resizebox{3.5cm}{!}{\begin{tikzpicture}[thick]
\def\r{2} 
\draw (0,0) circle(\r);

\shadedraw[left color=black!50!white, right color=black!10!white, draw=black!10!white] (-120:\r)--(70:\r)--(95:\r);

\shadedraw[left color=black!50!white, right color=black!10!white, draw=black!10!white] (-30:\r)--(120:\r)--(145:\r);

\path
(70:\r) node[above right]{$v_{k+2}$}
(95:\r) node[above=1mm]{$v_{k+1}$}
(145:\r) node[above left]{$v_{k-1}$}
(120:\r) node[above left]{$v_k$}
(-30:\r) node[below right]{$v_j$}
(-120:\r) node[below left]{$v_i$};

\foreach \point in {(-30:\r),(-120:\r),(70:\r),(95:\r),(145:\r),(120:\r)}
\draw[teal,fill=black] \point circle(.1);
\end{tikzpicture}}\qquad 
\resizebox{3.5cm}{!}{\begin{tikzpicture}[thick]
\def\r{2} 
\shadedraw[left color=black!50!white, right color=black!10!white, draw=black!10!white] (-120:\r)--(70:\r)--(95:\r);

\draw (0,0) circle(\r)
(120:\r)--(-5:\r)
(120:\r)--(-30:\r)
(120:\r)--(-55:\r);

\path
(70:\r) node[above right]{$v_{k+2}$}
(95:\r) node[above=1mm]{$v_{k+1}$}
(120:\r) node[above left]{$v_k$}
(-5:\r) node[below right]{$v_{j_1}$}
(-30:\r) node[below right]{$v_{j_2}$}
(-55:\r) node[below right]{$v_{j_3}$}
(-120:\r) node[below left]{$v_i$};

\foreach \point in {(-30:\r),(-120:\r),(70:\r),(95:\r),(120:\r),(-5:\r),(-55:\r)}
\draw[teal,fill=black] \point circle(.1);
\end{tikzpicture}} 
\caption{(Left) Bridges produced by $\{v_i,v_j\}$ where $v_i \to v_k \to v_j$. (Right) The same bridge $e=\{v_i,v_{k+1},v_{k+2}\}$ is produced by all of the pairs $\{v_i,v_{j_k}\}$ where $k=1,2,3$.}\label{fig:myfigure}
\end{figure}

\begin{lemma}\label{lem:main lem}
Suppose $\mathcal{C}$ is a maximal Berge cycle in $\cH$ with $|V(\mathcal{C})| \in \{n-1, n-2, n-3\}$.  The following holds for any $u \in V(\cH) - V( \mathcal{C})$ and usuable set $U_u$ with $|U_u| \geq d_{\cC}(u)/6$:
\begin{enumerate}[\normalfont(a)]
    \item $|U_u \cap S| \leq 3$.
    \item $|U_u \cap B|\geq |U_u|-3$.
    \item $|R(v_i,v_j)|\geq d_\cC(v_i)+d_\cC(v_j)-|V(\cC)|-4$ \  for all $v_i,v_j \in U_u$.
    \item $u \in S'$.
    \item $|U_u \cap B'|\leq 3$.
\end{enumerate}

\end{lemma}
\begin{proof}
We first prove (a). 
Suppose $U_u \cap S$ contains at least $4$ distinct points, $Z:=\{v_{i_1}, \ldots v_{i_4}\}$. Each of the ${4 \choose 2}$ pairs in $Z$ are adjacent by~\eqref{small}. Let $f_{j,k} \in E(v_{i_j}, v_{i_k})$. By Claim \ref{uadjacent}, for all $j,k \in [4]$, $f_{j,k} \subseteq \{e_{i_j}, e_{i_k}\}$.  
There are 4 edges in $\{e_{i_1}, \ldots, e_{i_4}\}$, each of which contains at most one pair in $Z$ (since the $v_{i_j}$'s are nonconsecutive in $\cC$ by definition of usable). That is, each $e_{i_t}$ corresponds to at most one edge $f_{j,k}$. This contradicts that there are ${4 \choose 2} > 4$ pairs in $Z$.



Item (b) follows directly from (a) since $|U_u|=|U_u \cap S|+|U_u \cap B|$.
Now we prove (c). Fix $v_i, v_j \in U_u$ and define $N'_{\cC}(v_i) = N_{\cC}(v_i) - \{v_j\}, N'_{\cC}(v_j) = N_{\cC}(v_j) - \{v_i\}$, \[I_1 = \{k: v_i \to v_k \to v_j \text{ and } v_{k} \in N'_{\cC}(v_i)\}, \; I_2 = \{k: v_j \to v_k \to v_i \text{ and } v_{k} \in N'_{\cC}(v_i)\},\]
\[J_1 = \{k: v_i \to v_k \to v_j \text{ and } v_{k} \in N'_{\cC}(v_j)\}, \; J_2 = \{k: v_j \to v_k \to v_i \text{ and } v_{k} \in N'_{\cC}(v_j)\}.\]

 Observe $|I_1| + |I_2| =|N'_{\cC}(v_i)| \geq d_{\cC}(v_i) - 1$ with equality only if $v_j \in N_{\cC}(v_i$). Similarly, $|J_1| + |J_2| = |N'_{\cC}(v_j)| \geq d_{\cC}(v_j)-1$. 

We have
\[|I_1^- \cap J_1| + |I_2 \cap J_2^-| \geq d_{\cC}(v_i) + d_{\cC}(v_j) - 2 - |I_1^- \cup J_1| -  |I_2 \cup J_2^-| \geq d_{\cC}(v_i) + d_{\cC}(v_j) - 2 - |V(\cC)|.\] 

Suppose $v_k \in I_2 \cap J_2^-$. (The case for $v_k \in I_1^- \cap J_1$ is symmetric so we omit it.) This implies $v_j \to v_k \to v_i$, $v_{k+1} \in N'_{\cC}(v_j)$, and $v_k \in N'_{\cC}(v_i)$. Note that $v_{k+1} \neq v_i$ and $v_{k} \neq v_j$ by the definitions of $N'_{\cC}(v_j), N'_{\cC}(v_i)$. 

Since $U_u$ is usable, there exists distinct edges $e(u, v_{i-1}) \in E(u,v_{i-1})-\{e_{i-2}\}$ and $e(u, v_{j-1}) \in E(u,v_{j-1})-\{e_{j-2}\}$.
Moreover, there also exists distinct edges $f_i \in E(v_i, v_k), f_j \in E(v_j, v_{k+1})$. Note that since $\cH$ is a $[3]$-graph and $k+1 \neq i$, we have $f_i,f_j,e(u, v_{i-1}),$ and $e(u, v_{j-1})$ are four distinct edges unless $f_j = e(u,v_{i-1})$ ($f_i \neq e(u, v_{j-1})$ since $v_j \to v_k \to v_{i}$). If $f_j = e(u,v_{i-1})$, we have $v_{k+1} = v_{i-1}$ and $f_j = e(u,v_{i-1})= \{u,v_{i-1}, v_j\} \notin E(\cC)$. Then we get a Berge cycle
\[v_{j-1}, e(v_{j-1}, u), f_j, v_j, e_j, v_{j+1}, e_{j+1}, v_{j+2}, \ldots, v_{j-1}\]
containing $V(\cC) \cup \{u\}$. This contradicts the maximality of $\cC$. 
Thus we may assume that the four edges are distinct.

If $f_i \notin \{e_{k-1}, e_{i}\}$ and $f_j \notin \{e_{k+1}, e_j\}$, then we obtain the Berge cycle  (see Figure \ref{fig:lem2path})
\[v_j, e_{j}, \ldots, v_k, f_i, v_i, e_{i}, v_{i+1}, \ldots, v_{j-1}, e(v_{j-1}, u), u, e(v_{i-1}, u), v_{i-1}, e_{i-2}, \ldots, v_{k+1}, f_j, v_j\]
containing $V(\cC) \cup \{u\}$ which again contradicts the maximality of $\cC$.

\begin{figure}[htb]
\centering
\resizebox{3.5cm}{!}{\begin{tikzpicture}[thick]
\def\r{2} 
\draw (0,0) circle(\r)
(120:\r)--(-30:\r)
(95:\r)--(-120:\r)
(0,0)--(-95:\r)
(0,0)--(-5:\r);

\path
(95:\r) node[above=1mm]{$v_{k+1}$}
(120:\r) node[above left]{$v_k$}
(-30:\r) node[below right]{$v_i$}
(-120:\r) node[below left]{$v_j$}
(0,0) node [above=1mm] (v1) {$u$}
(-95:\r) node[below=1mm]{$v_{j-1}$}
(-5:\r) node[right]{$v_{i-1}$};

\foreach \point in {(-30:\r),(-120:\r),(95:\r),(120:\r),(0,0),(-95:\r),(-5:\r)}
\draw[teal,fill=black] \point circle(.1);
\draw[red, ultra thick] (-95:\r) arc (-95:-33:2);
\draw[red, ultra thick]  (95:\r)arc (95:-5:2);
\draw[red, ultra thick](-1,1.7321) arc (119.9993:240:2);
\draw[red, ultra thick](95:\r)--(-120:\r);
\draw[red, ultra thick](120:\r)--(-30:\r);
\draw[red, ultra thick](0,0)--(-95:\r);
\draw[red, ultra thick](0,0)--(-5:\r);
\end{tikzpicture}}
\caption{$v_i, v_j \in U_u$, $f_i \notin \{e_{k-1},e_i\}$, and $f_j \notin \{e_{k+1},e_j\}$.}\label{fig:lem2path}
\end{figure}
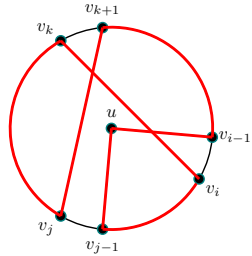

Hence $f_i \in \{e_{k-1}, e_{i}\}$ or  $f_j \in \{e_{k+1}, e_j\}$. We say $f_i$ is a {\em $(i,j;k)$-pseudo-bridge} if $f_i = e_i$, and similar for $f_j$ if $f_j=e_j$. Thus the pair $\{v_i, v_j\}$ produces either a $(i,j;k)$-bridge or a $(i,j;k)$-pseudo-bridge. Moreover, if say $f_i$ is a pseudo-bridge, then $f_i = e_i = \{v_i, v_{i+1}, v_k\}$, and hence $e_i$ is not a $(i,j;k')$-pseudo-bridge for any $k' \neq k$. 

Therefore at most two $v_k \in I_2 \cap J_2^-$ have $(i,j;k)$-pseudo-bridges (one for $e_i$ and one for $e_j)$. The remaining $v_k$ correspond to distinct $(i,j;k)$-bridges. Similar holds for $v_k \in I_1^- \cap J_1$. Altogether the pair $\{v_i,v_j\}$  produces at least $d_{\cC}(v_i) + d_{\cC}(v_j) - |V(\cC)| - 2 - 2$ bridges. This concludes the proof of (c). 
%



Next we will prove (d).  We count the number of bridges produced by pairs of big vertices in $U_u \cap B$. By Observation~\ref{bridgeproduced} (applied to $U_u \cap B$) and the fact $|V(\cC)| \leq n-1$, 

\begin{align*}
\bigg|\bigcup_{\{v_i, v_j\} \in U_u \cap B} R(v_i, v_j)\bigg| &\geq \frac{1}{|U_u \cap B|-1} \sum_{v_i,v_j \in UP_u\cap B} (d_{\cC}(v_i)+d_{\cC}(v_j)-|V(\cC)|-4)
\\
&\geq \frac{1}{|U_u \cap B|-1} \sum_{v_i,v_j \in U_u\cap B}(n+d_0 - 3 - 3 - (n-1) -4)\\
&\geq \frac{1}{|U_u \cap B|-1} \sum_{v_i,v_j \in U_u\cap B}(d_0-9)\\
&=\frac{(d_0-9)}{2}|U_u\cap B|.\\
\end{align*}

Suppose $u \notin S'$, i.e., 
$d_{\cC}(u) \geq \frac{n}{4}+\gamma_1$.
Since $|U_u \cap B| \geq |U_u|-3 \geq \frac{d_{\cC}(u)}{6}-3$, we have 

\[\frac{(d_0-9)}{2}|U_u\cap B|\geq \frac{d_0-9}{2}\left(\frac{d_{\cC}(u)}{6}-3\right) 
\geq \frac{d_0-9}{2}\left(\frac{n}{24}+\frac{\gamma_1}{6}-3\right) 
\geq n.\]

Here we assume 
\begin{constrain} \label{constrain1}
 $d_0\geq 57$ and $\gamma_1\geq 18$.
\end{constrain}

Therefore $|\bigcup_{\{v_i, v_j\} \in U_u \cap B} R(v_i, v_j)| \geq n$, contradicting~\eqref{numbridges}. Hence, $u \in S'$.

Finally we will prove (e). We count the number of bridges produced by pairs of super big vertices in $U_u \cap B'$. By Observation~\ref{bridgeproduced} (applied to $U_u \cap B'$) and the fact $|V(\cC)| \leq n-1$,

\begin{align*}
\bigg|\bigcup_{\{v_i, v_j\} \in U_u \cap B'} R(v_i, v_j)\bigg| &\geq \frac{1}{|U_u \cap B'|-1} \sum_{v_i,v_j \in U_u\cap B'} (d_{\cC}(v_i)+d_{\cC}(v_j)-|V(\cC)|-4)\\
&\geq \frac{1}{|U_u \cap B'|-1} \sum_{v_i,v_j \in U_u\cap B'}\left(\frac{3n}{4}+\gamma_2+ \frac{3n}{4}+\gamma_2 -(n-1)-4\right)\\
&\geq \frac{1}{|U_u \cap B'|-1} \sum_{v_i,v_j \in U_u\cap B'}\left(\frac{n}{2}+2\gamma_2-3\right)\\
&=\frac{|U_u \cap B'|}{2} \left( \frac{n}{2}+2\gamma_2-3\right).\\
\end{align*}

Suppose $|U_u \cap B'|>3$. Then we have,

\[\frac{|U_u \cap B'|}{2} \left( \frac{n}{2}+2\gamma_2-3\right)\geq \frac{4}{2}\left(\frac{n}{2}+2\gamma_2-3\right) \geq n.\]

Here we assume
\begin{constrain} \label{constrain2}
 $\gamma_2\geq \frac{3}{2}$.
\end{constrain}

Therefore $|\bigcup_{\{v_i, v_j\} \in U_u \cap B'} R(v_i, v_j)| \geq n$, contradicting~\eqref{numbridges}. Hence,
$|U_u \cap B'| \leq 3$.
\end{proof}

\section{Proof of Theorem \ref{mainthm}}

Let $\mathcal{C}$ be a longest Berge cycle and suppose $\cC$ is not Hamiltonian. Then $\cC$ is maximal and $|V(\cC)| \in \{n-1,n-2\}$ by Claim \ref{n-1n-2}. Let $u \in V(\cH) - V(\mathcal{C})$ and fix a usable set $U_u$ with $|U_u| \geq d_{\cC}(u)/6$, which exists by Lemma~\ref{lem: usableset}. Define \[B_1=B \setminus B'=\left\{v: \frac{n+d_0}{2}-3 \leq d_\cC(v) < \frac{3n}{4}+\gamma_2\right\}.\] By Lemma \ref{lem:main lem}(a) and (e), we have $|U_u \cap B_1|\geq |U_u|-6$.

Moreover, for any $u \in V(\cH) - V(\cC)$, Lemma~\ref{lem:main lem}(d) implies $ u \in S'(\cC)$, and so for any $v \in S(\cC) \cup B_1(\cC)$, 
\[d_{\partial\cH}(u) + d_{\partial\cH}(v) \leq d_{\cC}(u) + 2 + d_{\cC}(v) + 2 < \frac{n}{4} +\gamma_1+ 2+  \frac{3n}{4} +\gamma_2+ 2 \leq n+d_0.\]

Here we assume
\begin{constrain}\label{constrain3}
$\gamma_1+\gamma_2+4 \leq d_0$.
\end{constrain}

We obtain the following.
\begin{equation}\label{s'b'}
    \mbox{If $u \in V(\cH) - V(\cC) $ and $v \in S(\cC) \cup B_1(\cC)$, then $u$ and $v$ are adjacent.}
\end{equation}

We show the following fact to be used several times throughout the rest of the proof.
\begin{claim}\label{vivjcycle}
    For any $Y \subset V(\cC) - S'(\cC)$ with $|Y| \leq 2$, there does not exist a Berge cycle $\cC'$ with $V(\cC) \cup \{u\} - Y \subseteq V(\cC')$.
\end{claim}
\begin{proof}
If such a Berge cycle $\cC'$ exists, then $|V(\cC')| \geq |V(\cC)| + 1 - |Y| \geq n-3$. For any $y \in Y$, $d_{\cC'}(y) \geq d_{\cC}(y) -1$ with equality only if $y$ and $u$ are not adjacent. By~\eqref{s'b'}, $y \notin N(u)$ implies $y \in B'(\cC)$. Therefore in either case $d_{\cC'}(y) \geq n/4 + \gamma_1$, i.e., $y \notin S'(\cC')$. 

Lemma~\ref{lem:main lem} (d) implies $\cC'$ is not maximal. Then there exists a Berge cycle $\cC''$ with $V(\cC') \subsetneq V(\cC'')$. If $Y \subseteq V(\cC'')$, then $|V(\cC'')| > |V(\cC)|$, contradicting the choice of $\cC$ as a longest Berge cycle. Otherwise if say $y \in Y- V(\cC'')$, then again by Lemma~\ref{lem:main lem}(d), $\cC''$ is not maximal and we may find another Berge cycle $\cC'''$ which is longer than $\cC$. \end{proof}

Let $U_u' = U_u \cap B_1(\cC)$.  We consider the set $X = \{v_{i+1}: v_i \in U_u'\} (=U_u'^+)$. 

\begin{claim}\label{v_i+1}
    If $v_{i+1} \in X \cap N(u)$, then $E(u,v_{i+1}) = \{e_{i+1}\}$.
\end{claim}
\begin{proof} Suppose $v_{i+1} \in X \cap N(u)$. By~\eqref{s'b'} and the definition of usable, $u$ is adjacent to each of $v_{i-1}, v_i$.
    Suppose towards contradiction that $f \in E(v_{i+1}, u) - \{e_{i+1}\}$. 
    Let $e(u, v_{i-1}) \in E(u, v_{i-1}) - \{e_{i-2}\}$.
    
    If $f \neq e(u, v_{i-1})$, then the Berge cycle $\cC'= v_{i-1}, e(u,v_{i-1}), u, f, v_{i+1}, e_{i+1}, \ldots, v_{i-1}$ has vertex set $V(\cC)\cup \{u\} - \{v_i\}$. This contradicts Claim~\ref{vivjcycle}.

    Therefore $f=e(u, v_{i-1})$, i.e., $f = \{v_{i-1}, u, v_{i+1}\}$. 
 In particular, $f \notin E(\cC)$.
 Recall that $v_{i}$ is adjacent to $u$, and let $f' \in E(u,v_i)$. Let $\cC''$ be the Berge cycle obtained by replacing the segment $v_{i-1}, e_{i-1}, v_i, e_{i}, v_{i+1}$ in $\cC$ with $v_{i-1}, e_{i-1}, v_i, f', u, f, v_{i+1}$ if $f' \neq e_{i-1}$, or with $v_{i-1}, e_{i-1}, u, f, v_{i+1}$ if $f' = e_{i-1}$. 

 Then $\cC''$ is either longer than $\cC$, or $V(\cC'') = V(\cC)\cup \{u\} - \{v_i\}$. In the latter case,  $v_i \notin S'(\cC)$ which again contradicts Claim~\ref{vivjcycle}.
\end{proof}

\begin{claim}\label{uneigh}For every $v_{i+1} \in X$, either $v_{i+1} \in B'(\cC)$ or $v_{i+1} \in S'(\cC)$. \end{claim}

\begin{proof}
    Suppose $v_{i+1} \notin B'(\cC)$. By~\eqref{s'b'}, $v_{i+1}$ is adjacent to $u$. By Claim~\ref{v_i+1}, $E(u,v_{i+1}) = \{e_{i+1}\}$, i.e., $e_{i+1}=\{v_{i+1},v_{i+2},u\}$. If $v_{i+1} \notin S'(\cC)$, then the Berge cycle $v_{i-1}, e(u,v_{i-1}), u, e_{i+1}, v_{i+2}, e_{i+2}, \ldots, v_{i-1}$ has vertex set $V(\cC) \cup \{u\} - \{v_{i}, v_{i+1}\}$. This contradicts Claim~\ref{vivjcycle}.
\end{proof}

\begin{claim}\label{cla:S'}
$|X \cap S'(\cC)| \leq 1$.\end{claim}
\begin{proof}
    Suppose towards contradiction that $v_{i+1}, v_{j+1} \in X \cap S'(\cC)$. By~\eqref{small} and~\eqref{s'b'}, $v_{i+1}$ and $v_{j+1}$ are adjacent and are each adjacent to $u$.  Let $f \in E(v_{i+1}, v_{j+1})$, and fix distinct edges $e(u,v_{i-1}) \in E(u,v_{i-1}) - \{e_{i-2}\}$ and  $e(u,v_{j-1}) \in E(u,v_{j-1}) - \{e_{j-2}\}$. 
    By Claim~\ref{v_i+1}, $e_{i+1} = \{u,v_{i+1}, v_{i+2}\} \neq f$, and similar for $e_{j+1}$. 
    If $f = e(u, v_{j-1}) (= \{u, v_{j-1}, v_{j+1}\}$), so $j = i+2$, the Berge cycle \[v_{j-1}, f, u, e_{j-1}, v_j, e_{j}, \ldots, v_{j-1}\] is longer than $\cC$, a contradiction. Here we use the fact that $e_{j-1} = e_{i+1} = \{v_{i+1}, v_{i+2}, u\}$.
    We obtain a similar contradiction if $f = e(u, v_{i-1})$, so we may assume the three edges $f, e(u,v_{i-1}), e(u,v_{j-1})$ are distinct.
    The Berge cycle
    \[v_{i-1}, e(u,v_{i-1}), u, e(u,v_{j-1}), v_{j-1}, \ldots, v_{i+1}, f, v_{j+1}, e_{j+1}, \ldots, v_{i-1}\]

    has vertex set $V(\cC) \cup \{u\} - \{v_i, v_j\}$ (See Figure \ref{fig:proof1path}), contradicting Claim~\ref{vivjcycle}.

    \end{proof}
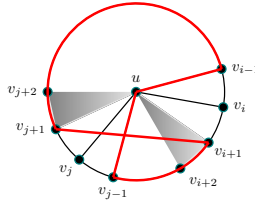
\begin{figure}[htb]
\centering
\resizebox{3.5cm}{!}{\begin{tikzpicture}[thick]
\def\r{2} 
\draw (0,0) circle(\r)
(0,0)--(-10:\r)
(0,0)--(-105:\r)
(0,0)--(-130:\r)
(0,0)--(-155:\r)
(0,0)--(-35:\r)
(0,0)--(15:\r);

\shadedraw[left color=black!50!white, right color=black!10!white, draw=black!10!white] (0,0)--(-155:\r)--(-180:\r);
\shadedraw[left color=black!50!white, right color=black!10!white, draw=black!10!white] (0,0)--(-35:\r)--(-60:\r);

\path
(0,0) node [above=1mm] (v1) {$u$}
(-35:\r) node[below right]{$v_{i+1}$}
(-130:\r) node[below left]{$v_j$}
(-10:\r) node[right=1mm]{$v_i$}
(-105:\r)  node[below=1mm]{$v_{j-1}$}
(15:\r)  node[right=1mm]{$v_{i-1}$}
(-155:\r)  node[left=1mm]{$v_{j+1}$}
(-180:\r)  node[left=1mm]{$v_{j+2}$}
(-60:\r)  node[below right=1mm]{$v_{i+2}$};

\foreach \point in {(0,0),(-35:\r),(-130:\r),(-10:\r),(-105:\r),(15:\r),(-155:\r),(-180:\r),(-60:\r)}
\draw[teal,fill=black] \point circle(.1);
\draw[red,ultra thick] (0,0)--(-105:\r);
\draw[red,ultra thick] (0,0)--(15:\r);
\draw[red,ultra thick] (-155:\r)--(-35:\r);
\draw[red,ultra thick] (15:\r) arc (15:205:2);
\draw[red,ultra thick] (-35:\r) arc (-35:-105:2);
\end{tikzpicture}}
\caption{$v_{i+1},v_{j+1} \in X \cap S'(\cC)$.}\label{fig:proof1path}
\end{figure}

Recall that for a pair of vertices $v_i$ and $v_j$, $R(v_i,v_j)$ is the set of bridges produced by the pair $\{v_i,v_j\}$. In the next claim, we count $R(v_{i+1}, v_{j+1})$ among pairs of vertices in $X$.
The proof is very similar to that of Lemma~\ref{lem:main lem}(c).

\begin{claim}\label{X-bridge}
    Let $v_{i+1}, v_{j+1} \in X$. Then $|R(v_{i+1}, v_{j+1})| \geq d_{\cC}(v_{i+1}) + d_{\cC}(v_{j+1}) -6$.
\end{claim}
\begin{proof}
 For $v_{i+1}, v_{j+1} \in X, $ define $N'_{\cC}(v_{i+1}) = N_{\cC}(v_{i+1}) - \{v_j, v_{j+1}\}, N'_{\cC}(v_{j+1}) = N_{\cC}(v_{j+1}) - \{v_i, v_{i+1}\}$, \[I_1 = \{k: v_{i+1} \to v_k \to v_{j+1} \text{ and } v_{k} \in N'_{\cC}(v_{i+1})\}, \; I_2 = \{k: v_{j+1} \to v_k \to v_{i+1} \text{ and } v_{k} \in N'_{\cC}(v_{i+1})\},\]
\[J_1 = \{k: v_{i+1} \to v_k \to v_{j+1} \text{ and } v_{k} \in N'_{\cC}(v_{j+1})\}, \; J_2 = \{k: v_{j+1} \to v_k \to v_{i+1} \text{ and } v_{k} \in N'_{\cC}(v_{j+1})\}.\]

 Note that $|I_1| + |I_2| =|N'_{\cC}(v_{i+1})| \geq d_{\cC}(v_{i+1}) - 2$ and similarly, $|J_1| + |J_2| = |N'_{\cC}(v_{j+1})| \geq d_{\cC}(v_{j+1})-2$. We have
\[|I_1^- \cap J_1| + |I_2 \cap J_2^-|\geq d_{\cC}(v_{i+1}) + d_{\cC}(v_{j+1}) - 4 - |I_1^- \cup J_1| -  |I_2 \cup J_2^-| \geq d_{\cC}(v_{i+1}) + d_{\cC}(v_{j+1}) - 4 - |V(\cC)|.\] 

Suppose $v_k \in I_2 \cap J_2^-$. (The case for $v_k \in I_1^- \cap J_1$ is symmetric so we omit it.) This implies $v_{j+1} \to v_k \to v_{i+1}$, $v_{k+1} \in N'_{\cC}(v_{j+1})$, and $v_k \in N'_{\cC}(v_{i+1})$. Note that $v_{k+1} \notin \{v_i,v_{i+1}\}$. 

By definition of usable set, there exists distinct edges $e(u, v_{i-1}) \in E(u,v_{i-1})-\{e_{i-2}\}$, $e(u, v_{j-1}) \in E(u,v_{j-1})-\{e_{j-2}\}$.
Moreover, there also exists distinct edges $f_{i+1} \in E(v_{i+1}, v_k)$ and $f_{j+1} \in E(v_{j+1}, v_{k+1})$. Note that since $\cH$ is a $[3]$-graph and $k+1 \neq i, i+1$, we have $f_{i+1},f_{j+1},e(u, v_{i-1}),$ and $e(u, v_{j-1})$ are distinct edges unless $f_{j+1} = e(u,v_{i-1})$ ($f_{i+1} \neq e(u,v_{j-1})$ since $v_{j+1} \to v_k \to v_{i+1}$). If $f_{j+1} = e(u,v_{i-1})$, we have $v_{k+1} = v_{i-1}$ and $f_{j+1} = \{u,v_{i-1}, v_{j+1}\} \notin E(\cC)$. Then we get a Berge cycle
\[v_{j-1}, e(u,v_{j-1}), u , f_{j+1}, v_{j+1}, e_{j+1}, \ldots, v_{j-1}\]
on the vertex set $V(\cC) \cup \{u\}-\{v_j\}$, contradicting Claim~\ref{vivjcycle}. 
Thus we may assume that the four edges are distinct.

If $f_{i+1} \notin \{e_{k-1}, e_{i+1}\}$ and $f_{j+1} \notin \{e_{k+1}, e_{j+1}\}$, then we obtain the Berge cycle (See Figure \ref{fig:proof2path})
\[v_{j+1}, e_{j+1}, \ldots, v_k, f_{i+1}, v_{i+1}, e_{i+1}, \ldots, v_{j-1}, e(u, v_{j-1}), u, e(u, v_{i-1}), v_{i-1}, e_{i-2}, v_{i-2}, \ldots, v_{k+1}, f_{j+1}, v_{j+1},\]
which has vertex set $V(\cC)\cup \{u\}-\{v_i,v_j\}$, again contradicting Claim~\ref{vivjcycle}. 
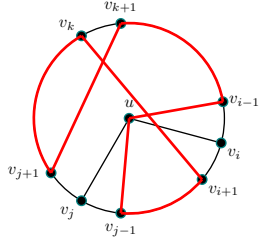
\begin{figure}[htb]
\centering
\resizebox{3.5cm}{!}{\begin{tikzpicture}[thick]
\def\r{2} 
\draw (0,0) circle(\r);
\draw (0,0)--(-120:\r);
\draw (0,0)--(-15:\r);

\path
(95:\r) node[above=1mm]{$v_{k+1}$}
(120:\r) node[above left]{$v_k$}
(-15:\r) node[below right]{$v_i$}
(-40:\r) node[below right]{$v_{i+1}$}
(-120:\r) node[below left]{$v_j$}
(-145:\r) node[left=1mm]{$v_{j+1}$}
(0,0) node [above=1mm] (v1) {$u$}
(-95:\r) node[below=1mm]{$v_{j-1}$}
(10:\r) node[right]{$v_{i-1}$};

\foreach \point in {(-15:\r),(-120:\r),(95:\r),(120:\r),(0,0),(-95:\r),(10:\r),(-145:\r),(-40:\r)}
\draw[teal,fill=black] \point circle(.1);
\draw[red, ultra thick] (10:\r) arc (10:95:2);
\draw[red, ultra thick]  (-145:\r)arc (-145:-240:2);
\draw[red, ultra thick]  (-40:\r) arc (-40:-95:2);
\draw[red, ultra thick](95:\r)--(-145:\r);
\draw[red, ultra thick](120:\r)--(-40:\r);
\draw[red, ultra thick](0,0)--(-95:\r);
\draw[red, ultra thick](0,0)--(10:\r);
\end{tikzpicture}}
\caption{$v_{i+1}, v_{j+1} \in X$, $f_{i+1} \notin \{e_{k-1},e_{i+1}\}$, and $f_{j+1} \notin \{e_{k+1},e_{j+1}\}$.}\label{fig:proof2path}
\end{figure}

Hence $f_{i+1} \in \{e_{k-1},e_{i+1}\}$ or $f_{j+1} \in \{e_{k+1},e_{j+1}\}$. We say $f_{i+1}$ is a {\em $(i+1,j+1;k)$-pseudo-bridge} if $f_{i+1} = e_{i+1}$, and similar for $f_{j+1}$ if $f_{j+1}=e_{j+1}$. Thus the pair $\{v_{i+1}, v_{j+1}\}$ produces either a $(i+1,j+1;k)$-bridge or a $(i+1,j+1;k)$-pseudo-bridge. If say $f_{i+1}$ is a pseudo-bridge, then $f_{i+1} = e_{i+1} = \{v_{i+1}, v_{i+2}, v_k\}$, and therefore $e_{i+1}$ is not a $(i+1,j+1;k')$-pseudo-bridge for any $k' \neq k$. 

Therefore at most two $v_k \in (I_2 \cap J_2^-) \cup (I_1^- \cap J_1)$ have $(i+1,j+1;k)$-pseudo-bridges. The remaining $v_k$ correspond to distinct $(i+1,j+1;k)$-bridges.  Altogether the pair $\{v_i,v_j\}$  produces at least $d_{\cC}(v_i) + d_{\cC}(v_j) - |V(\cC)| - 4 - 2$ bridges.
\end{proof}

\begin{claim}\label{cla:du}
$d_{\partial \cH}(u) \geq d_0+2$.    
\end{claim}

\begin{proof}
If $u$ is adjacent to all vertices, then $d_{\partial \cH}(u)=n-1 \geq d_0+3$. Otherwise suppose $v \in V(\cH)$ is not adjacent to $u$. We have $d_{\partial \cH}(u)+d_{\partial \cH}(v) \geq n+ d_0$. Since $d_{\partial \cH}(v) \leq n-2$, it follows $d_{\partial \cH}(u)\geq d_0+2$.   
\end{proof}

We are now ready to prove Theorem~\ref{mainthm}. Let $X' = \{v_{i+1} \in X: v_{i+1} \in B'(\cC)\}$. 
 By Claims~\ref{uneigh} and \ref{cla:S'}, $|X'| \geq |U_u'| - 1 = |U_u \cap B_1| - 1$. 
  We count the number of bridges produced by pairs of vertices in $X'$. By Claim \ref{X-bridge} and Observation~\ref{bridgeproduced} (applied to $X'$), 

\begin{align*}
|\bigcup_{\{v_{i+1}, v_{j+1}\} \in X'} R(v_{i+1}, v_{j+1})|  &\geq \frac{1}{|X'|-1} \sum_{\{v_{i+1}, v_{j+1}\} \in X'} (d_\cC(v_{i+1})+d_\cC(v_{j+1})-|V(\cC)|-6)\\
&\geq \frac{1}{|X'|-1} \sum_{\{v_{i+1}, v_{j+1}\} \in X'}\left(\frac{3n}{4}+\gamma_2+\frac{3n}{4}+\gamma_2-(n-1)-6\right)\\
&\geq \frac{1}{|X'|-1} \sum_{\{v_{i+1}, v_{j+1}\} \in X'}\left(\frac{n}{2}+ 2\gamma_2-5\right)\\
&\geq \frac{1}{|X'|-1} {|X'| \choose 2}\left(\frac{n}{2}+ 2\gamma_2-5\right).\\
&= \frac{|X'|}{2}\left(\frac{n}{2}+ 2\gamma_2-5\right)\\
& \geq \frac{(|U_u\cap B_1|-1)}{2}\left(\frac{n}{2}+ 2\gamma_2-5\right).\\
\end{align*}
Since
$|U_u \cap B_1| \geq |U_u|-6 \geq \frac{d_\cC(u)}{6}-6$
and $d_\cC(u)\geq d_{\partial \cH}(u)-1\geq d_0+1$,
we have

\[\frac{|U_u \cap B_1|-1}{2}\left(\frac{n}{2}+2\gamma_2-5\right)  \geq\frac{d_0-41}{12}\left(\frac{n}{2}+2\gamma_2-5\right) \geq n.\]

Here we assume
\begin{constrain}\label{constrain4}
 $d_0\geq 65$ and $\gamma_2\geq \frac{5}{2}$.
\end{constrain}

Now we choose $d_0=65$, $\gamma_1=18$, and $\gamma_2=13$ so that Constraints \ref{constrain0}, \ref{constrain00}, \ref{constrain1}, \ref{constrain2}, \ref{constrain3}, and \ref{constrain4},   are all satisfied.

Therefore $|\bigcup_{\{v_{i+1}, v_{j+1}\} \in X} R(v_{i+1}, v_{j+1})| \geq n$, contradicting~\eqref{numbridges}. This completes the proof. 
\qed

\section{Concluding remarks}
The $r$-uniform hypergraph $\cH'_{r}$ in Construction~\ref{const1} shows that for all $3 \leq r < n$, there exists non-Berge Hamiltonian $[r]$-graphs $\cH$ satisfying $d_{\partial \cH}(u) + d_{\partial \cH}(v) \geq n+r-3$ for all nonadajcent $u,v \in V(\cH)$. For $n$ large compared to $r$, we propose the following problem which generalizes Theorem~\ref{mainthm}.

\begin{question}\label{qr} Let $r \geq 3$ be a fixed integer. There exists $d_0 = d_0(r)$ and $n_0 
= n_0(r)$ such that for every $n \geq n_0$ the following holds: if $\cH$ is an $n$-vertex $[r]$-graph such that every pair of nonadjacent vertices $u, v \in V(\cH)$ satisfies
\[d_{\partial\cH}(u) + d_{\partial\cH}(v) \geq n + d_0,\] then $\cH$ contains a Hamiltonian Berge cycle. 
\end{question}

With our methods, even the case $r=4$ would be very difficult to prove.
Again by the example $\cH_{r}'$, $d_0 = r-2$ would be best possible. 
Lastly we note that $n_0 = \omega(r)$ is necessary due to the following example.

\begin{constr}Let $n, k \in \mathbb N$ such that ${k \choose 2}< n$ and $k \mid n$, and set $r = 2n/k$. Let $\cG_k$ be the $r$-graph with $V(\cG_{k}) = V_1 \cup V_2 \cup \ldots V_k$, $|V_i| = n/k$ for all $i \in [k]$, and $E(\cG_{k}) = \{V_i \cup V_j: i \neq j, i,j \in [k]\}$.
\end{constr}

One can view $\cG_{k}$ as a hypergraph blow-up of the clique $K_k$.
All vertices are adjacent, so $\cG_k$ trivially satisfies the Ore-condition in Question~\ref{qr} for any $d_0$. However $\cG_k$ has ${k \choose 2} < n$ edges, so it does not contain a Hamiltonian Berge cycle.


\begin{thebibliography}{99}
\small


\bibitem{BGHS}
J.-C. Bermond, A. Germa, M.-C. Heydemann, D. Sotteau: Hypergraphes Hamiltoniens,
in \emph{Probl\`emes combinatoires et th\`eorie des graphes} (Colloq. Internat.
CNRS, Univ. Orsay, Orsay, 1976). Colloq. Internat. CNRS,  \textbf{260} (CNRS,
Paris, 1978), 39--43.
%

\bibitem{CEP}
D. Clemens, J. Ehrenm\" uller, and Y. Person: A Dirac-type theorem for Hamilton
Berge cycles in random hypergraphs, {\em Electron. J. Combin.} \textbf{27} (2020), no. 3, Paper No. 3.39, 23 pp.


\bibitem{CP}
M. Coulson, G. Perarnau:
A Rainbow Dirac's Theorem, 
{\em SIAM J. Discrete Math.} \textbf{34} (2020), no. 3, 1670--1692.


\bibitem{D}
G. A. Dirac:
Some theorems on abstract graphs, \emph{Proc. London Math. Soc. (3)} \textbf{2}
(1952), 69--81.
%
%
%

\bibitem{EG}
P. Erd\H{o}s, T. Gallai:
On maximal paths and circuits of graphs, \emph{Acta Math. Acad. Sci. Hungar.} 
\textbf{10} (1959), 337--356.



\bibitem{EGMSTZ}  B. Ergemlidze, E. Gy\H ori, A. Methuku,  N. Salia, C. Tompkins, and O. Zamora:
    Avoiding long Berge cycles: the missing cases $k=r+1$ and $k=r+2$. {\em Combin. Probab. Comput.} \textbf{29} (2020),  423--435.
    

%
%
\bibitem{FKLdirac} Z. F\"uredi, A. Kostochka, R. Luo: Berge cycles in non-uniform hypergraphs,  {\em  Electronic J. Combin.}  \textbf{27} (2020), Paper No. 3.9, 13 pp.


%
%
%
%
%
%
%



 \bibitem{KL1} A. Kostochka and R. Luo:
 On $r$-uniform hypergraphs with circumference less than $r$,
{\em Discrete Appl. Math.} \textbf{276} (2020), 69--91.



\bibitem{KLM}
A. Kostochka, R. Luo, G. McCourt:
Dirac's Theorem for hamiltonian Berge cycles in uniform hypergraphs, 
{\tt  arXiv:2109.12637}, (2022), 23 pp.

\bibitem{LW} L. Lu and Z. Wang:
On Hamiltonian Berge cycles in $3$-uniform hypergraphs,
{\em Discrete Mathematics} \textbf{344} (2021), 112462.




%
%
%
\bibitem{MHG}
Y. Ma, X. Hou, J. Gao:
A Dirac-type theorem for uniform hypergraphs, 
{\tt  arXiv:2004.05073}, (2020), 17 pp.

\bibitem{Ore}
{\O}. Ore: Note on Hamilton circuits, {\em American Mathematical Monthly} {\bf 67} (1) (1960).



%
%
%
%


\end{thebibliography}
\end{document}